\documentclass[12pt]{amsart}

\usepackage{amssymb,amsmath,psfrag}
\usepackage{graphics,graphicx}

\newtheorem{lem}{Lemma}[section]   
\newtheorem{cor}[lem]{Corollary}

\newtheorem{thm}[lem]{Theorem}

\theoremstyle{remark}

\newcommand\A{{\it A}}
\newcommand\HP{\mathrm{HP}}

\begin{document}

\begin{center}

\Large{\bf Lindstr\"om's conjecture on a class of algebraically non-representable matroids}
\normalsize

{\sc Rigoberto Fl\'orez}

{\sc State University of New York at Binghamton \\
     Binghamton, NY, U.S.A.\ 13902-6000}



\end{center}

\bigskip

\footnotesize

{\it Abstract:} {Gordon introduced a class of matroids $M(n)$, for
prime $n\ge 2$,  such that $M(n)$ is algebraically representable,
but only in characteristic $n$. Lindstr\"om proved that $M(n)$ for
general $n\ge 2$ is not algebraically representable if $n>2$ is an
even number, and he conjectured that if $n$ is a composite number it
is not algebraically representable. We introduce a new kind of
matroid called {\it harmonic matroids}, of which full algebraic
matroids are an example. We prove the conjecture in this more
general case.} \normalsize

\thispagestyle{empty}

\vspace{0.2cm}
\section {introduction} An \emph{algebraic representation} over a field $F$ of a
matroid $M$ is a mapping $f$ from the elements of $M$ into a field
$K$, which contains $F$ as subfield, such that a subset $S$ is
independent in $M$ if and only if $|f(S)|=|S|$ and $f(S)$ is
algebraically independent over $F$.

Let $F$ and $K$ be two fields, such that $F \subseteq K.$ Assume
that $F$ and $K$ are algebraically closed and $K$ has finite
transcendence degree over $F$. Then those subfields of $K$ which are
algebraically closed and contain $F$ form the flats of a matroid. We
call this matroid the  $\emph{full algebraic matroid}$ $\A(K/F)$. A
set $S \subseteq K$ is independent in this matroid if and only if
its elements are algebraically independent over $F$. Therefore, a
full algebraic matroid $\A(K/F)$ of rank 3 consists of the field $F$
(the flat of transcendence degree zero over $F$), atoms (flats of
transcendence degree one over $F$), lines (flats of transcendence
degree two) and a plane, which is the field $K$ (since $K$ is of
transcendence degree 3 over $F$).

Gordon introduced a class of matroids $M(p)$ of rank 3 with $p \ge
2$ a  prime and proved in 1986 \cite{ga} that $M(p)$ is
algebraically representable. We state formally his result as
follows.

\begin{thm} [{\cite[Theorem 2]{ga}}] \label{gordon}
If $p$ is a prime, then $M(p)$ is algebraically representable over a
field $F$, when and only when the characteristic of $F$ is $p$.
\end{thm}

Lindstr\"om \cite{lhc} generalized the concept of \emph{harmonic
conjugate}  from projective geometry. He used this concept to show
that a matroid $M(n)$, defined for every $n \ge 2$, generalizing
Gordon's matroids, is not algebraically representable if $n$ is
even, and conjectured that if $n$ is a composite number then  $M(n)$
is not algebraically representable \cite{lcn}. We prove that this
conjecture is true.

\begin{thm} \label{main}
$M(n)$ is algebraically representable if and only if $n$ is a prime
number.
\end{thm}

Whoever is familiar with the \emph {complete  lift matroid}  $L_0
(\mathfrak{G} K_3)$ \cite {b3}, should be interested in the
following corollary.

\begin{cor} \label{submain}  Suppose that $\mathfrak{G}$ is a
finite abelian group. Then $L_0 (\mathfrak{G} K_3)$ is algebraically
representable over a field $F$ of characteristic $p$, when and only
when $\mathfrak{G}$ is a subgroup of the additive group of $F$.
\end{cor}

We are interested in $L_0 (\mathfrak{G} K_3)$, because $L_0
({\mathbb Z_n} K_3)$  contains $M(n)$. Any algebraic representation
of $M(n)$ extends to one of  $L_0 ({\mathbb Z_n} K_3)$ and the
representation of $L_0 ({\mathbb Z_n} K_3)$ is well understood
\cite{b4}.

The method of harmonic conjugation generalizes to representation in
many other matroids, including all projective planes that satisfy
the Little Desargues' theorem (a weaker form of Desargues' theorem)
and full algebraic matroids. Lindstr\"om used this method to
construct many algebraically nonrepresentable matroids. We will use
it here in many proofs. One can construct a projective plane in full
algebraic matroids using harmonic conjugation. The study of this
object is, however, left to another paper.

\section {Kinds of matroids}

We say that a matroid $H$ is a \emph {harmonic matroid} if whenever
a  configuration as in Figure \ref {f1} exists in $H$, then there
exists a point $x'$ in $H$ such that, for every configuration in $H$
of the form in Figure \ref {f2}, the points $p', r'$ and $x'$ are
collinear. The point $x'$ is called the \emph{harmonic conjugate} of
$x$ with respect to $y$ and $z$. If $x = x'$, we say that $x$ is
{\it self-conjugate }  with respect to $y$ and $z$.

We define  $\HP(y, x, z, p, q, r, s)$ to mean that the  points $y,
x, z, p, q, r, s$ form a submatroid of $H$ that has the
configuration in Figure \ref{f1}.

\begin{figure} [htbp]
 \centering
\includegraphics[scale=0.5]{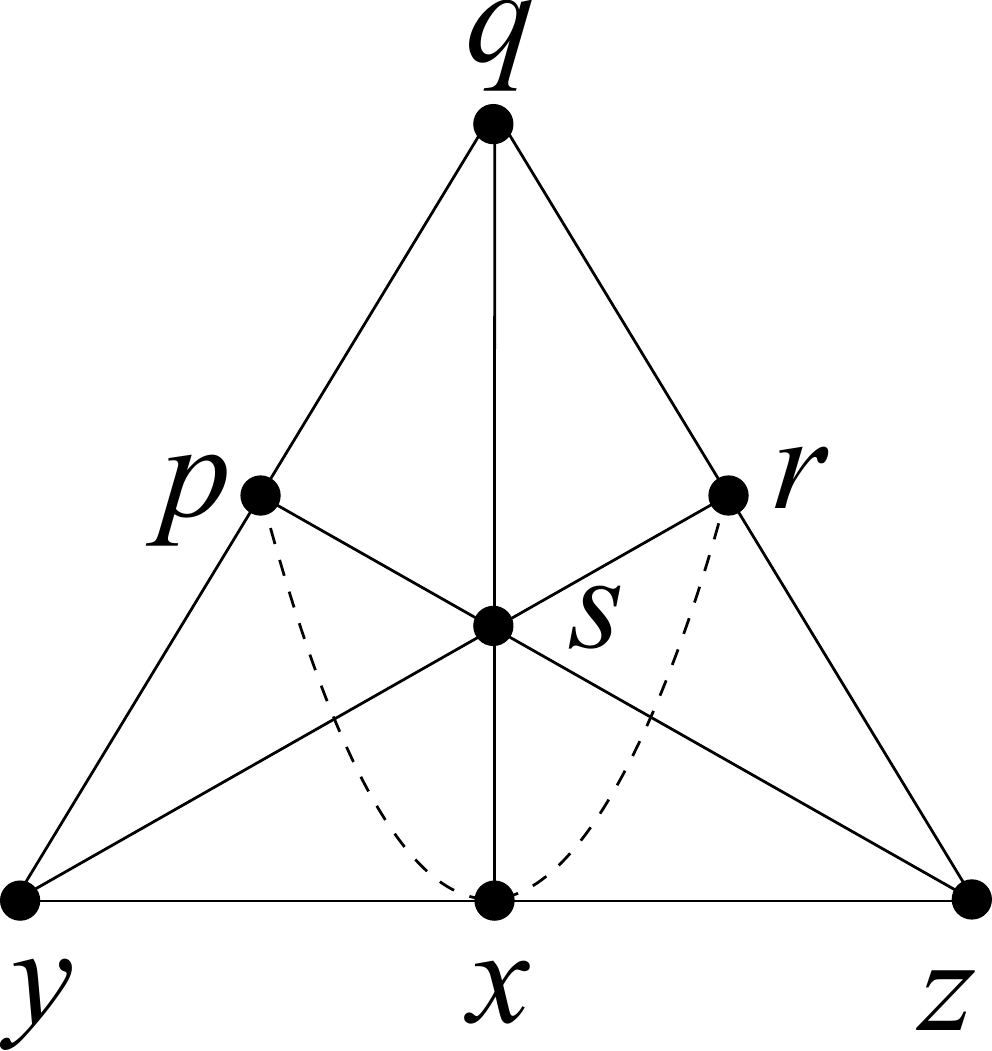}
\caption{The dashed line means that $p, x, r$ can be collinear.} \label{f1}
\end{figure}

\begin{figure} [htbp]
 \centering
\includegraphics[scale=0.5]{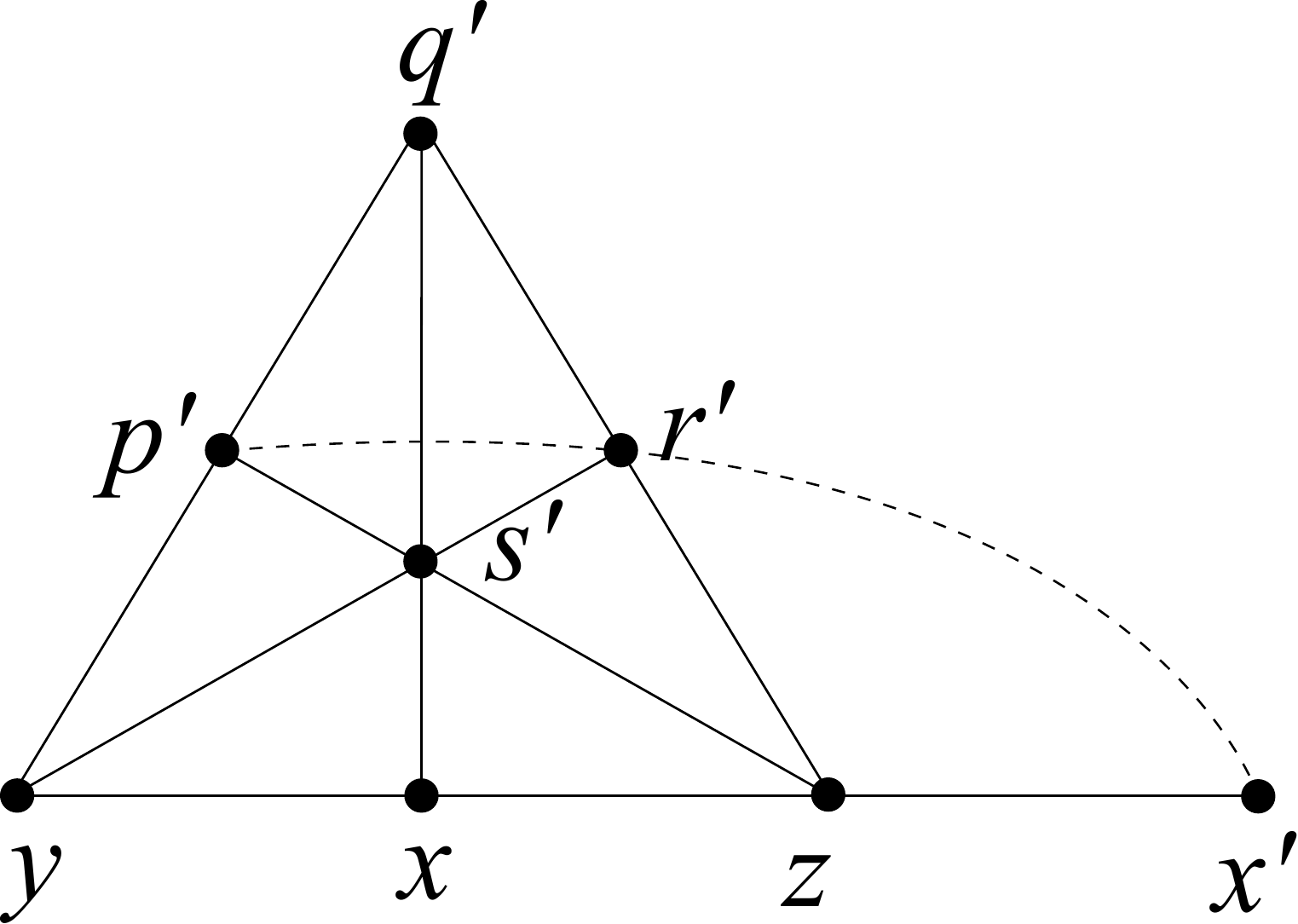}
\caption{ The dashed line means that $\{ p', r', x' \}$  is a new
line resulting from harmonic conjugation.} \label{f2}
\end{figure}

Let us introduce a couple of examples of harmonic matroids. A
Desargues  plane is an example of a harmonic matroid. More
generally, a little Desargues plane is a harmonic matroid (for
reference and definitions see \cite{st}), as shown in Lemma
\ref{little}.

A projective plane $\Pi$ satisfies the \emph {Fano axiom}, if the
Fano c onfiguration does not hold in $\Pi$. The uniqueness of the
harmonic conjugate gives a characterization of a little Desargues
plane. (The harmonic conjugate always exists in a projective plane,
but is not always unique.)

\begin{lem} [{\cite [Theorems  5.7.4 and 5.7.12] {st}}] \label{little}
A projective plane that satisfies Fano's axiom is a little Desargues
plane if and only if the harmonic conjugate is unique.
\end{lem}

The full algebraic matroid is the second example of  harmonic
matroids. Lindstr\"om proved in \cite{lhc} that, if the
configuration in Figure \ref{f1} holds in a full algebraic matroid
$\A(K/F)$, then the harmonic conjugate of $x$ with respect to $y$
and $z$ exists in $\A(K/F)$ and is unique. Lemma \ref {lindstrom} is
a restatement of this fact.

\begin{lem} [{\cite{lhc}}]\label{lindstrom}
A full algebraic matroid is a harmonic matroid.
\end{lem}

Now we will define the matroids  $M(n)$ for $n \ge 2$. $M(n)$  will
be a simple matroid  of rank 3. We need only specify the lines of
size of at least three. Any set of two points not belonging to the
same dependent line is also a line. The points are denoted by:
$$a_0,b_0,a_1,b_1,\dots , a_{n-1},b_{n-1}, c_0,c_1,d.$$
The dependent lines are
$$ \{ a_0,a_1, \dots , a_{n-1}, d \},\{ b_0,b_1, \dots ,b_{n-1}, d \}, \{ c_0,c_1,d \},\{ a_i,b_{i},c_0 \},\{ a_i,b_{i+1},c_1 \},$$
for $i=0,1, \dots, n-1$,  where  $b_n = b_0$.

The matroid $M(4)$ is depicted in Figure  \ref{f3}.

\begin{figure} [htbp]
 \centering
\includegraphics[scale=1]{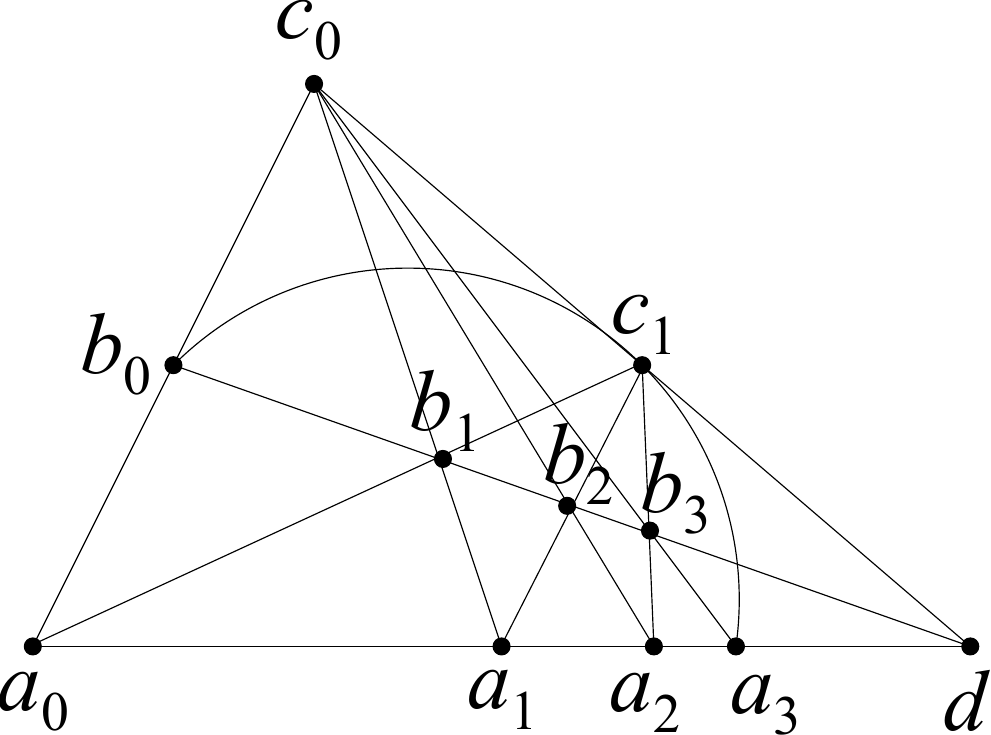}
\caption{ $M(4)$ } \label {f3}
\end{figure}

The Fano matroid (or Fano configuration) is the matroid in the
Figure \ref{f2} when $x=x'$. Note that $M(2)$ is the Fano matroid.
Lindstr\"om proved in \cite{lac} that the Fano matroid is
algebraically representable only in a field of characteristic 2. So
this proves Theorem \ref{main} for $n=2$.

Zaslavsky introduced in \cite {b3} the complete lift matroid for  a
group expansion graph. A \emph {group expansion graph }
$\mathfrak{G} K_3$ consists of a group $\mathfrak{G}$ and a \emph
{gain mapping} $\phi: \mathfrak{G} \times E \rightarrow
\mathfrak{G}$, from the edges of the underlying graph $\|
\mathfrak{G} K_3 \| = (V, \mathfrak{G} \times E)$ into a group
$\mathfrak{G}$, where $E = \{ e_1, e_2, e_3 \}$ is the set of edges
of the complete graph $K_3$. An edge $(g,e)$ has the same vertices
as $e$. One arbitrarily chooses an orientation of $e$, say
$(e;v,w)$, to define the gain $\phi ((g,e;v,w)) = g$; then the gain
of the reversed edge will be $\phi ((g,e;w,v)) = (\phi
(g,e;v,w))^{-1} = g^{-1}$. We call a circle (the edge set of a
simple closed path) \emph{balanced} if its edges' gains taken in
circular order multiply to the identity and \emph{unbalanced} if
otherwise (for reference and notations see \cite{b3} page 43 and
examples 3.6, 5.8). This gain graph gives rise to a matroid called
the \emph {complete lift matroid}, $L_0 (\mathfrak{G} K_3)$, where
the ground set is formed by the edges of $\mathfrak{G} K_3$ and an
extra point $D$. A subset of  $L_0 (\mathfrak{G} K_3)$ is dependent
if and only if it contains a balanced  circle, or two unbalanced
circles, or $D$ and an unbalanced circle.

An explicit description of the complete lift matroid $L_0(\mathbb
Z_n K_3)$  for a group expansion graph ${ \mathbb Z_n} K_3$, where $
n \ge 2$, is given as follows. Let $e_1, e_2, e_3$ be the edges of
$K_3$, oriented so that head$(e_1) = $ tail$(e_3)$ and head$(e_2) =
$ head$(e_3)$, and let $A_i := (i, e_1)$,  $B_i:=(i, e_2)$, $C_i:=
(i, e_3)$ be the edges of the graph $ ||{\mathbb Z_n} K_3||$ for $i
= 0, 1, \dots , n-1$, such that the gains are given by: $\phi (A_i)
= i$,  $\phi (B_i) = i$, $\phi (C_i) = i$ for $i = 0, 1, \dots ,
n-1$. So, the dependent lines of the complete lift matroid
$L_0(\mathbb Z_n K_3)$ are given by
$$\{ A_0,A_1, \dots , A_{n-1}, D \},\{ B_0,B_1, \dots , B_{n-1}, D \},\{ C_0,C_1,\dots , C_{n-1}, D \},\{ A_i,B_{i+j},C_j \},$$
where  $i,j = 0,1, \dots, n-1$ and $i+j$ is  modulo $n$. Any set of
two points not belonging to the same dependent hyperplane is also a
line.

The matroid $L_0(\mathbb Z_3 K_3)$ is depicted in Figure \ref{f4}.

\begin{figure} [htbp]
 \centering
\includegraphics[scale=1]{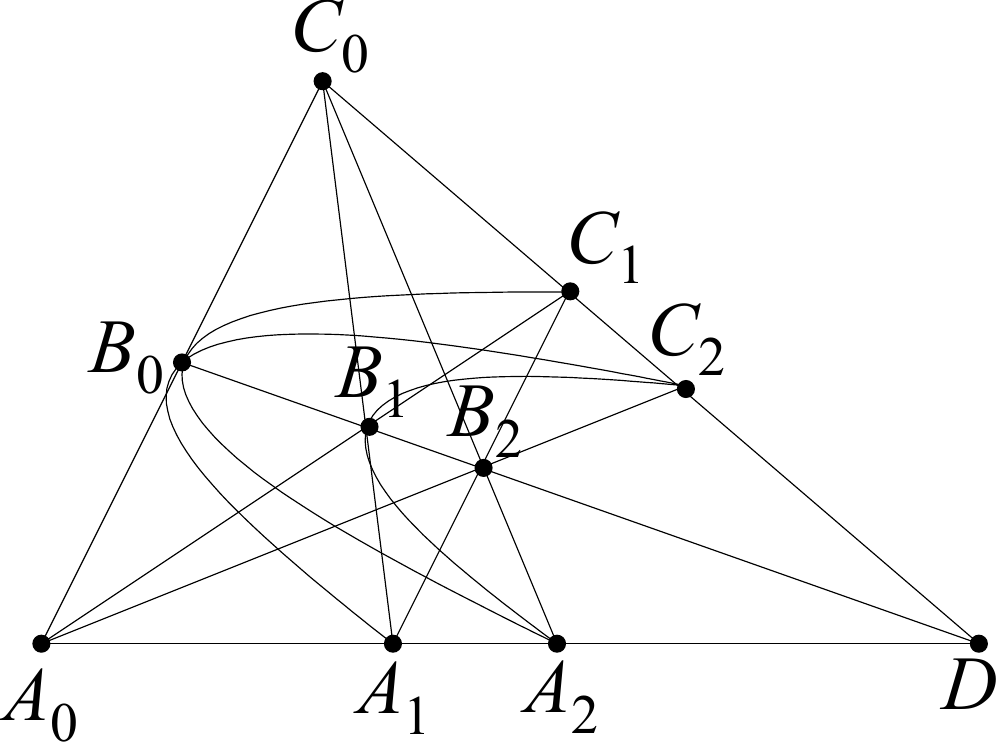}
\caption{ $L_0(\mathbb Z_3 K_3) $} \label{f4}
\end{figure}

\section {Embedding in harmonic matroids}

In this section we give a characterization for embeddability  of
$L_0 (\mathfrak{G} K_3)$ in a harmonic matroid $H$. We will use this
characterization and the fact that $M(n)$ is submatroid of $L_0
({\mathbb Z_n} K_3)$, to show when  $M(n)$ is or is not embeddable
in a harmonic matroid. This allows to say when $M(n)$ is or is not
algebraically representable.

\begin{lem} \label{nonembedding} If ${\mathbb Z_{p^n}}$ is a subgroup
of $\mathfrak{G}$, where $n \geq 2 $, then $L_0 (\mathfrak{G} K_3)$
does not embed in any harmonic matroid $H$.
\end{lem}

\begin{proof} Assume $L_0 (\mathfrak{G} K_3)$ embeds in some
harmonic matroid $H$. $L_0 ({\mathbb Z_{p^n}} K_3)$ is a submatroid
of $L_0 (\mathfrak{G} K_3)$, because  ${\mathbb Z_{p^n}}$ is a
subgroup of $\mathfrak{G}$. The subscripts will be calculated modulo
$p^n$.

For $t \in \{ 1, \dots , p -1 \}$, let $P(t)$ be the statement:
there  exists a point $ x_{t+2} \in H$, such that, for all  $i = 0,
1, \dots , p^n -1$, $ \{ A_i, B_{(t+1)i}, C_{t i}, x_{t+2} \}$ are
sets of collinear points, and also $ A_0,B_0,C_0,x_{t+2}$ are
collinear points.

First we prove $P(1)$. Since $\HP(A_0, B_0, C_0, A_i, D, C_i, B_i) $
holds for $ i=1, \dots , p^n -1 $, then there exists a point $x_3$
in $H$ (the harmonic conjugate of $ B_0 $ with respect to $A_0 $ and
$ C_0 $), such that $ \{A_i,C_i,x_3 \} $ are sets of collinear
points for $i= 1, \dots ,p^n -1$. This and the definition of $L_0
({\mathbb Z_{p^n}} K_3)$ imply that $\{A_i,B_{2i},C_i,x_3 \}$ are
sets of collinear points. This proves that $P(1)$ is true.

Now we suppose that for a fixed  $t \in \{ 1, \dots , p -2 \}$, and
for  all  $s = 1, \dots , t$, $P(s)$ is true, and we deduce
$P(t+1)$.

From $P(t-1)$ we can see that,  for all $j=0, 1, \dots , p^n -1$, $
\{ A_{j}, B_{tj}, C_{(t-1)j},x_{t+1} \}$ are sets of collinear
points. And from  $P(t)$ we can see that, for all $j=0, 1, \dots ,
p^n -1 $, $ \{ A_{j}, B_{(t+1)j}, C_{tj},x_{t+2} \}$ are sets of
collinear points. So, taking $j = ((t+1)/t)i $, the sets $\{
A_{((t+1)/t)i},  C_{(t+1)i} , x_{t+2} \}$,  $\{ A_{((t+1)/t)i},
B_{(t+1)i}, x_{t+1} \}$ are formed by collinear points. This, the
inductive hypothesis, and the definition of  $L_0 ({\mathbb Z_{p^n}}
K_3)$ imply that
$$ \HP(A_0, x_{t+1}, x_{t+2}, C_{(t+1)i}, B_{(t+1)i}, A_i, A_{((t+1)/t)i} )$$
holds, for $ i=1, \dots , p^n -1 $. Therefore, we obtain a point
$x_{t+3}$,  the harmonic conjugate of $x_{t+1}$ with respect to
$x_{t+2}$ and $A_0$, such that  $\{A_i,B_{(t+2) i},C_{(t+1)
i},x_{t+3} \}$ are sets of collinear points and
$\{A_0,B_0,x_{t+2},x_{t+3} \}$ is a set of collinear points.  This
proves that $P(t+1)$ is true.

In particular, $P(p-1)$ provides that $\{A_i,B_{p i},C_{(p-1)
i},x_{p+1} \}$ are sets of collinear points for $i =1, 2, \dots, p^n
-1$. Now,  by taking $i=p^{n-1}$, the points
$A_{p^{n-1}},B_0,C_{(p-1) p^{n-1}},x_{p+1} $ are collinear (because
$ p^n \equiv 0 \text { modulo } p^n$). And using the fact that $
A_0,B_0,x_{p+1}$ are collinear points, we get that $A_0,
A_{p^n-1},B_0$ are collinear, which is a contradiction by definition
of $L_0 ({\mathbb Z_{p^n}} K_3)$.
\end{proof}

\begin{lem} \label{nonembedding2} If  ${\mathbb Z_p}$ and
${\mathbb Z_q}$ are subgroups of $\mathfrak{G}$, where $p < q$ are
primes, then $L_0 (\mathfrak{G} K_3)$ is not embeddable in any
harmonic matroid $H$.
\end{lem}

\begin{proof}
Since ${\mathbb Z_p}$ and ${\mathbb Z_q}$ are subgroups  of
$\mathfrak{G} $, then by definition of a group expansion graph, we
can see that ${\mathbb Z_p}K_3 $ and ${\mathbb Z_q} K_3$ are
subgraphs of  $\mathfrak{G} K_3$, with three edges in common (the
edges whose gains are the identity of the group). Therefore $L_0
({\mathbb Z_p} K_3)$ and $L_0 ({\mathbb Z_q} K_3)$ are submatroids
of  $L_0 (\mathfrak{G} K_3)$. Let 
$$E_1 = \{a_0,  b_0, c_0, a_1, b_1, c_1, \dots , a_{p-1}, b_{p-1}, c_{p-1} \}$$ and  
$$E_2 = \{a_0,  b_0, c_0,  A_1, B_1,  C_1, \dots , A_{q-1}, B_{q-1}, C_{q-1}\}$$ 
be two subsets of the edges of  $\mathfrak{G} K_3$, such that
$E_1$ and $E_2$ correspond with the edges of ${\mathbb Z_p} K_3$ and
${\mathbb Z_q} K_3$ respectively, where 
$$\{a_0, a_1, \dots, a_p, A_1, \dots , A_{q-1}\}, \{b_0, b_1, \dots , b_p, B_1, \dots
, B_{q-1}\} \text{ and } \{c_0, c_1, \dots , c_p, C_1, \dots , C_{q-1}\}$$
are sets of parallel edges. Suppose that the gains are $\phi (a_j) =
j$,  $\phi (b_j) = j$, $\phi (c_j) = j$,  $\phi (A_i) = i$,  $\phi
(B_i) = i$, $\phi (C_i) = i$.

Suppose that $L_0 (\mathfrak{G} K_3)$ embeds in some harmonic
matroid $H$.  The subscripts with letter $j$ will be calculated
modulo $p$ and the subscripts with letter $i$ will be calculated
modulo $q$.

For $t \in \{ 1, \dots , p -1 \}$, let $P(t)$ be  the statement:
there exists a point $ x_{t+2} \in H$, such that, for all $i \in \{
1, \dots , q -1 \}$ and for all $j \in \{ 0, 1, \dots , p -1 \}$,
the following sets are formed by collinear points: $ \{ A_i,
B_{(t+1)i}, C_{t i}, x_{t+2} \}$, $ \{ a_j, b_{(t+1)j}, c_{t j},
x_{t+2} \}$, and  also $ a_0,b_0,c_0, x_{t+2}$ are collinear points.

First we prove $P(1)$. Since $\HP(a_0, b_0, c_0, A_i, D, C_i, B_i) $
holds for $ i=1, \dots , q-1 $, there exists a point $x_3$ in $H$,
the harmonic conjugate of $ b_0 $ with respect to $a_0 $ and $ c_0
$, such that $ A_i,C_i,x_3 $ are collinear points for $i= 1, \dots
,q-1$. This and the definition of $L_0 ({\mathbb Z_q} K_3)$ imply
that $\{A_i,B_{2i},C_i,x_3 \}$ are sets of collinear points. Since
the harmonic conjugate is unique and  $\HP(a_0, b_0, c_0, a_j, D,
c_j, b_j) $ holds for $ j=1, \dots , p-1 $, then $\{a_j,c_j,x_3 \} $
are sets of collinear points. This proves that $P(1)$ is true.

Now we suppose that for a fixed  $t \in \{ 2, \dots , p -2 \}$, and
for all  $s = 1, \dots , t$, $P(s)$ is true, and we deduce
$P(t+1)$.

From $P(t-1)$ we can see that, for all $r=1, \dots , q-1$, $\{
A_{r}, B_{tr}, C_{(t-1)r}, x_{t+1} \}$  are sets of collinear
points. And from $P(t)$ we can see that, for all $r=1, \dots , q-1$,
$\{ A_{r},B_{(t+1)r}, C_{tr},x_{t+2} \}$ are sets of collinear
points. So, taking $r = ((t+1)/t)i $, the sets  $\{ A_{((t+1)/t)i},
B_{(t+1)i}, x_{t+1} \}$, $\{ A_{((t+1)/t)i}, C_{(t+1)i}, x_{t+2} \}$
are formed by collinear points. This, the inductive hypothesis, and
the definition of $L_0 ({\mathbb Z_q} K_3)$ imply that
$$ \HP(a_0, x_{t+1}, x_{t+2}, C_{(t+1)i}, B_{(t+1)i}, A_i, A_{((t+1)/t)i} )$$
holds, for $i = 1, \dots , q-1$. Therefore, we obtain a point
$x_{t+3}$, the harmonic conjugate  of $x_{t+1}$ with respect to
$x_{t+2}$ and $a_0$, such that $\{A_i,B_{(t+2) i},C_{(t+1)
i},x_{t+3} \}$ are sets of collinear points.

Similarly, from $P(t-1)$ we can see that, for all $k=1, \dots ,
p-1$,  $\{ a_{k}, b_{tk},c_{(t-1)k},x_{t+1} \}$ are sets of
collinear points. And from $P(t)$ we can see that, for all $k=1,
\dots , p-1$, $ \{ a_{k}, b_{(t+1)k}, c_{tk},x_{t+2} \}$ are sets of
collinear points. So, taking $k = ((t+1)/t)j $, the sets  $\{
a_{((t+1)/t)j}, b_{(t+1)j}, x_{t+1} \}$, $\{ a_{((t+1)/t)j},
c_{(t+1)j}, x_{t+2} \}$ are formed by collinear points. This, the
inductive hypothesis, and the definition of $L_0 ({\mathbb Z_p}
K_3)$ imply that
$$ \HP(a_0, x_{t+1}, x_{t+2}, c_{(t+1)j}, b_{(t+1)j}, a_j, a_{((t+1)/t)j} )$$
holds for $ j=1, \dots , p -1$. Then $\{a_j,b_{(t+2) j},c_{(t+1)
j},x_{t+3} \}$  are collinear points (because the harmonic conjugate
is unique). This proves that $P(t+1)$ is true.

In particular, $P(p-1)$ provides that $\{
A_i,B_{pi},C_{(p-1)i},x_{p+1} \}$,  $\{a_j,b_{pj},c_{(p-1)j},x_{p+1}
\}$ are sets of collinear points, for $j= 1, \dots , p-1$ and $i= 1,
\dots , q-1$. Taking $i=j = 1$, $\{ A_1,B_{p},C_{(p-1)},x_{p+1} \}$,
$\{a_1,b_{p},c_{(p-1)},x_{p+1} \}$ are sets of collinear points.
Here $ b_p = b_0 $, so $ x_{p+1}=  b_0 $. This implies that
$A_1,B_{p},C_{(p-1)},b_0$ are collinear points, which is a
contradiction by definition of $L_0 (\mathfrak{G} K_3)$.
\end{proof}

\begin{thm} \label{L0embedding} If $\mathfrak{G}$ is a
finite abelian group, then the following statements are equivalent.

(a) $L_0 (\mathfrak{G} K_3)$ embeds in some harmonic matroid.

(b) $\mathfrak{G} = ({\mathbb Z_p})^n$, where $p$ is a prime  number
and $n \geq 0$.

(c) $L_0 (\mathfrak{G} K_3)$ is algebraically representable.
\end{thm}

\begin{proof} If  $\mathfrak{G} \neq ({\mathbb Z_p})^n$, then by
Lemmas \ref {nonembedding} and  \ref {nonembedding2}, $L_0
(\mathfrak{G} K_3)$ does not embed in any harmonic matroid.

Zaslavsky in \cite{b4} proved that $L_0 ( ({\mathbb Z_p})^n K_3)$ is
linear over a field $F$, if $({\mathbb Z_p})^n$ is the additive
subgroup of $F$. Since every linear matroid over a field $F$ is
algebraic over $F$ \cite [Theorem 6.7.10] {oxley}, then $L_0 (
({\mathbb Z_p})^n K_3)$ is algebraically representable over $F$.

That (c) implies (a) follows from Lemma \ref {lindstrom}.
\end{proof}

\begin{thm} \label{embedfg} If $\mathfrak{G}$ is a finite group
such that $L_0 (\mathfrak{G} K_3)$ embeds in some harmonic matroid,
then  $\mathfrak{G}$ has exponent $p$ for some $p$ prime.
\end{thm}

\begin{proof} The proof follows by Lemmas  \ref{nonembedding} and \ref{nonembedding2}.
\end{proof}

\emph {Question}: If $\mathfrak{G}$ is a  finite group with exponent
$p$, not isomorphic to $({\mathbb Z_p})^n$ for $n \geq 1$, is it
possible for  $L_0 (\mathfrak{G} K_3)$ to be embeddable in some
harmonic matroid?

Now we compare the embeddability of $M(n)$ with the  embeddability
of $ L_0 ({\mathbb Z_n} K_3)$.

\begin{lem} \label{mnembedding} Any embedding of $M(n)$ in a harmonic
matroid $H$ extends uniquely to an embedding of $L_0 ({\mathbb Z_n}
K_3)$ in $H$.
\end{lem}

\begin{proof} Identifying $a_i, b_i, c_j, \text { and } d$ in $M(n)$
with $A_i$, $B_i$, $C_j$ and $ D$ in  $L_0 (\mathbb Z_n K_3)$,
respectively, for $i = 0, 1, \dots , n-1$ and $j= 0, 1$, we can see
that $M(n)$ is a submatroid of $L_0 ({\mathbb Z_n} K_3)$. The
subscripts will be calculated modulo $n$.

For all $j \in \{ 1, \dots , n-1 \}$, let $P(j)$ be the statement:
there exists  a point $ c_j \in H$, such that, for all  $i = 0, 1,
\dots , n -1$,  $a_i, b_{i+j}, c_j$ are collinear points, and  $c_0,
c_1 , c_j$ are collinear points.

By definition of $L_0 ({\mathbb Z_n} K_3)$, we can see  that $P(1)$
is true.

Now we suppose that for a fixed  $j \in \{ 1, \dots , n -2 \}$, and
for all  $s = 1, \dots , j$, $P(s)$ is true. This and the definition
of $M(n)$ imply that $ \HP(c_{j}, c_{j-1}, D, a_i, b_{i+j},
b_{i+j+1},a_{i+1})$ holds, for $i= 0, 1, \dots , n-1$. Therefore
there exists a point $c_{j+1}$ (the harmonic conjugate of  $c_{j-1}$
with respect to $c_j$ and $D$), such that  $a_i, b_{i+j+1}, c_{j+1}$
are collinear points. By uniqueness of the harmonic conjugate,
$a_i,b_{i+j},c_j $ are collinear points for $i,j=0,1, \dots n-1$.
This proves that $P(j+1)$ is true.

Let $M'(n)$ be the extension of $M(n)$ by $\{ c_2, \dots , c_{n-1}
\}$,  where the dependent lines are given by definition of $M(n)$
and by the above procedure.

Identify the $c_j$ in $M'(n)$ with  $C_j$ in  $L_0 ({\mathbb Z_n}
K_3)$,  for $j = 0,1, \dots , n-1$. We can see that the dependent
lines $\{a_i,b_{i+j},c_j \}$  and  $\{c_0,c_1, \dots , c_{n-1} \}$
are in one-to-one correspondence with the dependent lines
$\{A_i,B_{i+j},C_j \}$ and  $\{C_0,C_1, \dots , C_{n-1} \}$ in $L_0
({\mathbb Z_n} K_3)$, respectively, for  $i,j = 0,1, \dots , n-1$.
Now, it is easy to see the correspondence between all other flats of
$M'(n)$ and the flats of $L_0 ({\mathbb Z_n} K_3)$. So $M(n)$
extends to $L_0 ({\mathbb Z_n} K_3)$ in $H$.

The uniqueness holds because the harmonic conjugate is unique.
\end{proof}

Theorem \ref {gordon} and Lemma \ref {lindstrom} imply that $M(p)$
embeds in a harmonic matroid $H$ if $p$ is prime. As a natural
question we ask: Does $M(n)$ embed in a harmonic matroid, if $n$ is
composite? The answer to this question is stated formally in the
following theorem.

\begin{thm} \label{main1}
If $n$ is composite then $M(n)$ does not embed in any harmonic matroid.
\end{thm}

\begin{proof}
Assume that $M(n)$ embeds in a harmonic matroid $H$. By Lemma \ref
{mnembedding},  the embedding of  $M(n)$ extend to one of $L_0
({\mathbb Z_n} K_3)$. So the conclusion follows by Theorem \ref
{L0embedding}.
\end{proof}

The Lindstr\"om conjecture follows from Theorem \ref {gordon} and
Theorem \ref {main1},  as is shown in the following proof.

\begin{proof}[Proof of Theorem  \ref{main}]
By Theorem \ref {main1}, $M(n)$ is not embeddable in a harmonic
matroid when $n$ is  composite, therefore $M(n)$ is not
algebraically representable.

From Theorem \ref{gordon}, if $n$ is a prime number, then  $M(n)$ is
algebraically  representable. This completes the proof.
\end{proof}

\begin{proof}[Proof of Corollary  \ref{submain}]
If  $\mathfrak{G}$ is a subgroup of $ F^+ $, then $L_0 (\mathfrak{G}
K_3)$  is linear over $F$ \cite {b4}. Since every linear matroid
over a field is algebraic over the same field {\cite [Theorem
6.7.10] {oxley}}, then $L_0 (F^+ K_3)$ is algebraically
representable over $F$.

Since the characteristic of $F$ is $p$ then $L_0 ( {\mathbb Z_p}
K_3)$  is a submatroid of  $L_0 (F^+ K_3)$. By Lemma \ref
{mnembedding}, $M(p)$ is a submatroid of $L_0 ( {\mathbb Z_p} K_3)$.
So the conclusion follows  by Theorem \ref {gordon}.
\end{proof}

\section {Characteristic sets}

Define the \emph {linear characteristic set}, $\chi _{L} (M)$, to
be the set of field characteristics over which the matroid $M$ is
linear. Similarly, define the \emph {algebraic characteristic set},
$\chi _{A} (M)$, to be the set of field characteristics over which
the matroid $M$ is algebraic. The reader can find in \cite [Section
6.7, exercise 7 (ii)] {oxley} that, if $0 \in \chi _L (M)$, then
$\chi _{A} (M) = \{ 0, 2,3,5 , \dots \}$.

Corollary \ref {acset} shows that the point $D$ in the definition of
$ L_0 ( {\mathbb Z_n} K_3) $ is crucial for the characteristic set.
The \emph {lift matroid} $L ({\mathbb Z_n}  K_3)$ \cite{b2}  of the
group expansion graph $ {\mathbb Z_n}  K_3$, is  $L_0 ( {\mathbb
Z_n} K_3) \setminus D$, for $n \in {\mathbb Z}_{>0}$.

\begin{cor} \label{acset} Let  $L ( {\mathbb Z_n} K_3)$ and  $L_0 ( {\mathbb Z_n} K_3)$
be the lift matroid and complete lift matroid of $ {\mathbb Z_n}
K_3$, for  $ n \in {\mathbb Z}_{>0}$. Then

(a) $\chi _{L} ( L ( {\mathbb Z_n} K_3) ) =  \chi _{L} ( M(n)
\setminus d ) = \{p \text { prime} : p \nmid n  \} \cup \{ 0 \}$,
for $n \in {\mathbb Z}_{>0}$.

\vspace{0.2cm} (b) $\chi _{A} ( L ( {\mathbb Z_n} K_3) ) =  \chi
_{A} ( M(n) \setminus d ) = \{0, 2,3,5 , \dots \}$, for $ n \in
{\mathbb Z}_{>0}$.

\vspace{0.2cm}
(c) $ \chi _{A} ( L_0 ( {\mathbb Z_n} K_3) ) = \chi _{A} ( M(n) )= \left\{\begin{array}{ll}
                                                                    \{n \} & \mbox{ if $n$ is prime,} \\
                                                                  \emptyset  & \mbox{ if $n$ is composite.}
                                                       \end{array}
                                               \right.
   $
\end{cor}

\begin{proof} [Proof of (a) and (b)]
Since ${\mathbb Z}_n$  is a subgroup of ${\mathbb C} ^*$, and thus
by  \cite [Theorem 2.1] {b4}, $ L ( {\mathbb Z_n} K_3) $ is linearly
representable over ${\mathbb C}$, it follows that $0 \in  \chi _L (L
( {\mathbb Z_n}  K_3))$. So, $\chi _{A} (L ( {\mathbb Z_n} K_3)) =
\{ 0,2,3,5 , \dots \}$.

Let $p \in  \{p  \text { prime}: n \mid p^m - 1, \text{ for some } m
\in {\mathbb Z}_{>0} \}$  = $\{ p \text { prime}: p \nmid n  \}$.
There exists $m \in {\mathbb Z}_{>0}$ such that $n \mid p^m - 1$.
So, $n$ divides the order of the group $ {\mathbb Z}_{p^m -1}$ =
$GF(p^m)^* $. Therefore ${\mathbb Z}_n$ is a subgroup of $
GF(p^m)^*$, thus   $\{ p \text { prime}:  p \nmid n \}$  $\subseteq
\chi _{L} ( L ( {\mathbb Z_n} K_3) )$.

Now, suppose that there exists a prime $p$ such that  $\forall m \in
{\mathbb Z}_{>0}$, $ n \nmid p^m -1$. It follows that ${\mathbb
Z}_n$ $\nleq$  $GF(p^m)^*$. Then  $L ( {\mathbb Z_n} K_3) $ is not
linearly representable over  $GF(p^m)$. So $ p \notin \chi _{L} ( L
( {\mathbb Z_n} K_3))$.
\end{proof}

\begin{proof} [Proof of (c)]
The proof follows by Theorem \ref {L0embedding} and  Corollary \ref
{submain}.
\end{proof}


\begin{thebibliography}{9}
\bibitem{lcn}
B. Lindstr\"om,
A class of non-algebraic matroids of rank three.
\emph{Geometriae Dedicata} {\bf 23} (1987), 255--258.

\bibitem{lac}
B. Lindstr\"om,
On the algebraic characteristic set for a class of matroids.
\emph{Proceedings of the American Mathematical Society} {\bf 95} (1985), 147--151.

\bibitem{lhc}
B. Lindstr\"om,
On harmonic conjugates in full algebraic combinatorial geometries.
\emph{European Journal of Combinatorics} {\bf 7} (1986), 259--262.

\bibitem{ga}
G. Gordon,
Algebraic characteristic sets of matroids.
\emph{ Journal of Combinatorial Theory Series B} {\bf 44} (1988), 64--74.

\bibitem{oxley}
J.G. Oxley,
\emph{Matroid Theory}.
Oxford University Press Inc., New York, 1992.

\bibitem{st}
F.W. Stevenson,
\emph{Projective Planes}.
W.H. Freeman and Company, San Francisco, 1972.

\bibitem{b1}
T.~Zaslavsky,
Biased graphs. I. Bias, balance, and gains.
\emph{Journal of Combinatorial Theory Series B} {\bf 47} (1989), 32--52.

\bibitem{b2}
T.~Zaslavsky,
Biased graphs. II. The three matroids.
\emph{Journal of Combinatorial Theory Series B} {\bf 51} (1991), 46--72.

\bibitem{b3}
T.~Zaslavsky,
Biased graphs. III. Chromatic and dichromatic invariants.
\emph{Journal of Combinatorial Theory Series B} {\bf 64} (1995), 17--88.

\bibitem{b4}
T.~Zaslavsky,
Biased graphs. IV: Geometrical realizations.
\emph{Journal of Combinatorial Theory Series B} {\bf 89} (2003), 231--297.



\end{thebibliography}
\end{document}